\documentclass[reqno, 12pt]{amsart}
\usepackage{amsfonts}
\usepackage{mathrsfs}
\usepackage{amssymb,url}
\usepackage{graphicx}

\date{}

\allowdisplaybreaks[4]
\theoremstyle{plain}
\newtheorem{thm}{Theorem}[section]

\newtheorem{lem}[thm]{Lemma}

\newtheorem{rem}{Remark}

\theoremstyle{definition}

\newtheorem{Def}{Definition}
\newtheorem{prop}{Proposition}
\newtheorem{cla}{Claim}

\usepackage[colorlinks,linkcolor=blue,urlcolor=red,linktocpage]{hyperref}

\numberwithin{equation}{section}

\begin{document}

\title
[~~~]{On the functional Blaschke-Santal\'{o} inequality}

\author[Youjiang Lin]{Youjiang Lin}
\address{School of Mathematical Sciences, Peking University, Beijing, 100871, China;
Department of Mathematics, Department of Mathematics, Shanghai University, Shanghai, 200444, China}
 \email{\href{mailto: YOUJIANG LIN
<lxyoujiang@126.com>}{lxyoujiang@126.com}}
\author[Gangsong Leng]{Gangsong Leng}
\address{Department of Mathematics, Shanghai University, Shanghai, 200444, China} \email{\href{mailto:
Gangsong Leng <gleng@staff.shu.edu.cn>}{gleng@staff.shu.edu.cn} }

\begin{abstract} In this paper, using functional Steiner symmetrizations, we show
that Meyer and Pajor's proof of the Blaschke-Santal\'{o} inequality
can be extended to the functional setting.
\end{abstract}



\thanks{The authors would like to acknowledge the support from China Postdoctoral Science
Foundation Grant 2013M540806, National Natural Science Foundation of China under grant
11271244 and National Natural Science Foundation of China under grant 11271282 and the
973 Program 2013CB834201.}

\maketitle

\section{Introduction}

For a convex body $K\subset \mathbb{R}^n$ and a point
$z\in\mathbb{R}^n$, the polar body $K^z$ of $K$ with respect to $z$
is the convex set defined by $ K^{z}=\{y\in \mathbb{R}^n: \langle
y-z,x-z\rangle\leq 1\;{\rm for}\;{\rm every}\;x\in K\}$. The {\it
Santal\'{o} point} $s(K)$ of $K$ is a point for which
$V_n(K^{s(K)})=\min_{z\in int(K)}V_n(K^{z})$, where $V_n(K)$ denotes
the volume of set $K$.
 The Blaschke-Santal\'{o}
inequality \cite{Bl85,Sa49, Sc93} states that
$V_n(K)V_n(K^{s(K)})\leq V_n(B_2^n)^2$, where $B_2^n$ is the
Euclidean ball.

 For a log-concave function $f:\mathbb{R}^n\rightarrow
[0,\infty)$ and a point $z\in\mathbb{R}^n$, its polar with respect
to $z$ is defined by $f^{z}(y)=\inf_{x\in
\mathbb{R}^n}\frac{e^{-\langle x-z,y-z\rangle}}{f(x)}$. The {\it
Santal\'{o} point} $s(f)$ of $f$ is the point $z_0$ satisfying $\int
f^{z_0}=\inf_{z\in\mathbb{R}^n}\int f^{z}$.

The {\it functional Blaschke-Santal\'{o} inequality} of log-concave
functions is the analogue of Blaschke-Santal\'{o} inequality  of
convex bodies.
\begin{thm}\label{a1} (Artstein, Klartag, Milman). Let $f:\mathbb{R}^n\rightarrow [0,+\infty)$
be a log-concave function such that $0<\int f<\infty$. Then,
$\int_{\mathbb{R}^n} f\int_{\mathbb{R}^n} f^{s(f)}\leq(2\pi)^n$ with
equality holds exactly for Gaussians.
\end{thm}

When $f$ is even, the functional Blaschke-Santal\'{o} inequality
follows from an earlier inequality of  Ball \cite{Ba86}; and in
\cite{Fr07}, Fradelizi and Meyer proved something more general (see
also \cite{Le09}). Lutwak and Zhang \cite{Lu97} and Lutwak et al.
\cite{Lu04} gave other very different forms of the
Blaschke-Santal\'{o} inequality.  In this paper, we give a more
general result than Theorem \ref{a1}, which becomes into a special
case of $\lambda=1/2$ in Theorem \ref{2}.
\begin{thm}\label{2}
Let $f:\mathbb{R}^n\rightarrow[0,+\infty)$ be a log-concave function
such that $0<\int f<\infty$. Let $H$ be an affine hyperplane and let
$H_+$ and $H_-$ denote two closed half-spaces bounded by $H$. If
$\lambda\in (0,1)$ satisfies $\lambda
\int_{\mathbb{R}^n}f=\int_{H_+}f$. Then there exists $z\in H$ such
that
\begin{eqnarray}\label{12}
\int_{\mathbb{R}^n}f\int_{\mathbb{R}^n}f^z\leq
\frac{1}{4\lambda(1-\lambda)}(2\pi)^n.
\end{eqnarray}
\end{thm}

 In \cite{Leh09}, Lehec proved a very general functional version for non-negative Borel functions, Theorem \ref{2} is a particular case of result of Lehec.
 Lehec's proof is by induction
on the dimension, and the proof is by functional Steiner
symmetrizations. In fact, Mayer and Pajor \cite{Me90} have proved
the Blaschke-Santal\'{o} inequality for convex bodies, here we show
that Meyer and Pajor's proof of the Blaschke-Santal\'{o} inequality
can be extended to the functional setting.  It has recently come to
our attention that in a remark of \cite{Fr07}, Fradelizi  and Meyer
expressed the same idea to prove the functional Blaschke-Santal\'{o}
inequality.

\section{Notations and background materials}
Let $|\cdot|$ denote the Euclidean norm. Let ${\rm int} A$ denote
the interior of $A\subset\mathbb{R}^n$. Let ${\rm cl} A$ denote the
closure of $A$. Let ${\rm dim} A$ denote the dimension of $A$. A set
$C\subset\mathbb{R}^n$ is called a {\it convex cone} if $C$ is
convex and nonempty and if $x\in C$, $\lambda\geq0$ implies $\lambda
x\in C$. We define $C^{\ast}:=\{x\in\mathbb{R}^n:\langle
x,y\rangle\leq0\;\;{\rm for}\;{\rm all}\;y\in C\}$ and call this the
{\it dual cone} of $C$.

For a non-empty convex set $K\subset\mathbb{R}^n$ and an affine
hyperplane $H$ with unit normal vector $u$, the {\it Steiner
symmetrization} $S_H K$ of $K$ with respect to $H$ is defined as
$S_H K:=\{x^{\prime}+\frac{1}{2}(t_1-t_2)u:\;x^{\prime}\in
P_H(K),\;t_i\in I_K(x^{\prime})\;{\rm for}\;i=1,2\}$, where
$P_H(K):=\{x^{\prime}\in H:\;x^{\prime}+tu\in K\; {\rm for}\;{\rm
some}\;t\in\mathbb{R}\}$ is the projection of $K$ onto $H$ and
$I_K(x^{\prime}):=\{t\in\mathbb{R}:\;x^{\prime}+tu\in K\}$.

Let $\bar{\mathbb{R}}=\mathbb{R}\cup\{-\infty,\infty\}$. For a given
function $f:\mathbb{R}^n\rightarrow\bar{\mathbb{R}}$ and for
$\alpha\in\bar{\mathbb{R}}$ we use the abbreviation
$\{f=\alpha\}:=\{x\in\mathbb{R}^n:f(x)=\alpha\}$, and
$\{f\leq\alpha\}$, $\{f<\alpha\}$ etc. are defined similarly. A
function $f:\mathbb{R}^n\rightarrow\bar{\mathbb{R}}$ is called {\it
proper} if $\{f=-\infty\}=\emptyset$ and
$\{f=\infty\}\neq\mathbb{R}^n$. A function $\phi$ is called {\it
convex} if $\phi$ is proper and $\phi(\alpha
x+(1-\alpha)y)\leq\alpha \phi(x)+(1-\alpha)\phi(y)$ for all
$x,y\in\mathbb{R}^n$ and for any $0\leq\lambda\leq1$. A function $f$
is called {\it log-concave} if $f=e^{-\phi}$, where $\phi$ is a
convex function. A function
$f:\mathbb{R}^n\rightarrow\mathbb{R}\cup\{+\infty\}$ is called {\it
coercive} if $\lim_{|x|\rightarrow+\infty}f(x)=+\infty$.
 A function $f$
is called {\it symmetric} about $H$ if for any $x^{\prime}\in H$ and
$t\in\mathbb{R}$, $f(x^{\prime}+tu)=f(x^{\prime}-tu)$. A function
$f: \mathbb{R}^n\rightarrow\mathbb{R}$ is called {\it unconditional}
about $z$ if $f(x_1-z_1,\dots,
x_n-z_n)=f(|x_1-z_1|,\dots,|x_n-z_n|)$ for every
$(x_1,\dots,x_n)\in\mathbb{R}^n$. If $z=0$, then $f$ is called {\it
unconditional}.

The {\it effective domain} of convex function $\phi$ is the nonempty
set ${\rm dom}\phi:=\{\phi<\infty\}$. The {\it support} of function
$f$ is the set ${\rm supp}f:=\{f\neq 0\}$. For log-concave function
$f=e^{-\phi}$, it is clear that ${\rm supp}f={\rm dom}\phi$. The
nonempty set ${\rm
epi}\phi:=\{(x,r)\in\mathbb{R}^n\times\mathbb{R}:\;r\geq \phi(x)\}$
denote the {\it epigraph} of convex function $\phi$.

For an affine subspace $G$ of $\mathbb{R}^n$, let $G^{\perp}$ denote
the {\it orthogonal complement} of $G$, we have $G^{\bot}=\{x\in
\mathbb{R}^n: \langle x, y-y^{\prime}\rangle=0\; {\rm for\;
every\;}y, y^{\prime}\in G\}$. The {\it Santal\'{o} point}
$s_{G}(f)$ of $f$ about $G$ is a point satisfying $\int
f^{s_{G}(f)}=\inf_{z\in G}\int f^z$. Let $f$ be a log-concave
function such that $0<\int f<\infty$, and let $H_+$ and $H_-$ be two
half-spaces bounded by an affine hyperplane $H$; let $0<\lambda<1$;
we shall say that $H$ is {\it $\lambda$-separating} for $f$ if
$\int_{H_+}f\int_{H_-}f=\lambda(1-\lambda)\left(\int_{\mathbb{R}^n}f\right)^2$
and when $\lambda=1/2$, we shall say that $H$ is {\it medial} for
$f$. For a function $\phi: \mathbb{R}^n\rightarrow
\bar{\mathbb{R}}$, its Legendre transform about $z$ is defined by
$\mathcal {L}^{z}\phi(y)=\sup_{x\in\mathbb{R}^n}[\langle
x-z,y-z\rangle-\phi(x)]$. If $f(x)=e^{-\phi(x)}$, where $\phi(x)$ is
a convex function, then $\label{b22}f^{z}(y)=e^{-\mathcal
{L}^{z}\phi(y)}$. Since $\mathcal {L}^z(\mathcal {L}^z\phi)=\phi$
for a convex function $\phi$, $(f^z)^z=f$. If $z=0$, we shall use
the simpler notation $\mathcal {L}$ for $\mathcal {L}^{0}$.

Given two functions $f,g:\mathbb{R}^n\rightarrow [0,\infty)$, their
{\it Asplund product} is defined by $(f\star
g)(x)=\sup_{x_1+x_2=x}f(x_1)g(x_2)$. The {\it $\lambda$-homothety}
of a function $f$ is defined as $(\lambda\cdot
f)(x)=f^{\lambda}(\frac{x}{\lambda})$. Then, the classical
Pr\'{e}kopa inequality (see Pr\'{e}kopa \cite{Pr71,Pr73}) can be
stated as follows: Given $f,g:\mathbb{R}^n\rightarrow[0,+\infty)$
and $0<\lambda<1$, $\int(\lambda\cdot f)\star((1-\lambda)\cdot
g)\geq\left(\int f\right)^{\lambda}\left(\int g\right)^{1-\lambda}$.
The following lemma, as a particular case of a result due to Ball
\cite{Ball88}, was proved by
  Meyer and Pajor in \cite{Me90}.
\begin{lem}\label{74}\cite{Me90}
Let $f_0$, $f_1$, $f_2: \mathbb{R}^{+}\rightarrow\mathbb{R}^{+}$ be
three functions such that
$0<\int^{+\infty}_{0}f_i<\infty,\;i=0,1,2$, they are continuous and
suppose that $f_0\left(\frac{2xy}{x+y}\right)\geq
f_1(x)^{\frac{y}{x+y}}f_2(y)^{\frac{x}{x+y}}$ for every $x,y>0$.
Then one has
$$\frac{1}{\int^{+\infty}_{0}f_0(t)dt}\leq\frac{1}{2}\left(\frac{1}
{\int^{+\infty}_{0}f_1(t)dt}+\frac{1}{\int^{+\infty}_{0}f_2(t)dt}\right).$$
\end{lem}

\section{The functional Steiner symmetrization}

The familiar definition of Steiner symmetrization for a nonnegative
measurable function $f$ can be stated as following (see
\cite{Br74,Bu97,Bu09,Fo10}):
\begin{Def}\label{29}
For a measurable function $f:\mathbb{R}^n\rightarrow[0,+\infty)$ and
an affine hyperplane $H\subset\mathbb{R}^n$, let $m$ denote the
Lebesgue measure, if $m(\{f>t\})<+\infty$ for all $t>0$, then its
Steiner symmetrization is defined as
\begin{eqnarray}\label{1}
S_H f(x)=\int_{0}^{\infty}\mathcal {X}_{S_H \{f>t\}}(x)dt,
\end{eqnarray}
where $\mathcal {X}_{A}$ denotes the characteristic function of set
$A$.
\end{Def}
Next, we give a approach of defining Steiner symmetrization for
coercive convex functions by the Steiner symmetrization of
epigraphs. A similar functional steiner symmetrization is defined in
a remark of AKM's paper \cite{Ar04} and studied in an article by
Lehec \cite{Le08}. The idea of our definition is same as the given
definition in a remark at the end of an article by Fradelizi and
Meyer \cite{Fr07}.

\begin{Def}\label{19}
For a coercive convex function $\phi$ and an affine hyperplane
$H\subset\mathbb{R}^n$, we define the {\it Steiner symmetrization}
$S_H \phi$ of $\phi$ with respect to $H$ as a function satisfying
\begin{eqnarray}\label{17}
{\rm epi}(S_H \phi)=S_{\widetilde{H}}({\rm cl}\;{\rm epi} \phi),
\end{eqnarray}
where
$\widetilde{H}=\{(x^{\prime},s)\in\mathbb{R}^{n+1}:x^{\prime}\in
H\}$ is an affine hyperplane in $\mathbb{R}^{n+1}$.
\end{Def}
\begin{rem}

(i) By Definition \ref{19}, for an integrable log-concave function
$f=e^{-\phi}$, the Steiner symmetrization of $f$ can be defined as
$S_H f:=e^{-(S_H \phi)}$. If we define $S_H f$ by Definition
\ref{29}, then $S_H f$ still satisfies (\ref{17}). Thus, for
integrable log-concave functions, the two definitions are
essentially same.

(ii) By Definition \ref{19}, for a given $x^{\prime}\in H$ and any
$s\in\mathbb{R}$, we have $V_1\left(\{(S_H
\phi)(x^{\prime}+tu)<s\}\right)=V_1\left(\{\phi(x^{\prime}+tu)<s\}\right)$.
By the Fubini's theorem, we have
\begin{eqnarray}\label{21}\int_{\mathbb{R}}(S_H f)(x^{\prime}+tu)dt
=\int_{\mathbb{R}}f(x^{\prime}+tu)dt.\end{eqnarray} Similarly,
$\int_{\mathbb{R}^n} S_H f=\int_{\mathbb{R}^n} f$ is also
established.
\end{rem}

\begin{prop}\label{7} For a coercive convex function
$\phi$ and an affine hyperplane $H\subset\mathbb{R}^n$ with outer
unit normal vector $u$, then $S_H \phi$ has the following
properties.

(i) $S_H \phi$ is a closed coercive convex function and symmetric
about $H$.

(ii) Let $H_1$ and $H_2$ be two orthogonal hyperplanes in
$\mathbb{R}^n$, then $S_{H_2}(S_{H_1} \phi)$ is symmetric about both
$H_1$ and $H_2$.

(iii) For any given $x^{\prime}\in H$ and $t\in\mathbb{R}$, let
$\phi_1(t):=\phi(x^{\prime}+tu)$ and $(S\phi_1)(t):=(S_H
\phi)(x^{\prime}+tu)$, then $(S\phi_1)(t)$ satisfies one of the
following three cases. {\it 1)}.
$(S\phi_1)(t)=\phi_1(t_1)=\phi_1(t_1-2t)$ for some
$t_1\in\mathbb{R}$. {\it 2)}. $(S\phi_1)(t)=\phi_1(t_0-2t)\geq
\lim_{t\rightarrow t_0,\;t<t_0}\phi_1(t)$ for some
$t_0\in\mathbb{R}$. {\it 3)}. $(S\phi_1)(t)=\phi_1(t_0+2t)\geq
\lim_{t\rightarrow t_0,\;t>t_0}\phi_1(t)$ for some
$t_0\in\mathbb{R}$.
\end{prop}

\begin{proof}
(i)  By the fact that $\phi$ is convex if and only if ${\rm
epi}\phi$ is convex, since $\phi$ is convex, ${\rm epi} \phi$ is a
convex subset of $\mathbb{R}^{n+1}$. Since the closure of a convex
set is convex, and the Steiner symmetrization of a convex set is
also convex, by (\ref{17}), ${\rm epi}(S_H \phi)$ is a convex subset
of $\mathbb{R}^{n+1}$. Therefore, $S_H \phi$ is a convex function.
By Definition \ref{19}, it is clear that $S_H \phi$ is closed,
coercive and symmetric with respect to $H$.

(ii) Since ${\rm epi}(S_{H_2}(S_{H_1} \phi))$ is symmetric about
both $\widetilde{H_1}$ and $\widetilde{H_2}$, where
$\widetilde{H_i}=\{(x^{\prime},s)\in\mathbb{R}^{n+1}:x^{\prime}\in
H_i\}$ ($i=1,2$),  $S_{H_2}(S_{H_1} \phi)$ is symmetric about both
$H_1$ and $H_2$.

(iii) If ${\rm dom}\phi_1=\mathbb{R}$, by (\ref{17}) in Definition
\ref{19}, we have
\begin{eqnarray}\label{33}
{\rm epi}(S\phi_1)=S_{\widetilde{H}}({\rm cl}\;{\rm epi}\phi_1).
\end{eqnarray} Thus there exists some $t_1\in\mathbb{R}$ satisfying
\begin{eqnarray}\label{4}
(S\phi_1)(t)=\phi_1(t_1)=\phi_1(t_1-2t).
\end{eqnarray}

If ${\rm dom}\phi_1\neq\mathbb{R}$, then there exist eight cases for
${\rm dom}\phi_1$: 1) $[\alpha,\beta]$; 2) $(\alpha,\beta)$; 3)
$(\alpha,\beta]$; 4) $[\alpha,\beta)$; 5) $(-\infty,\beta]$; 6)
$(-\infty,\beta)$; 7) $[\alpha,+\infty)$; 8) $(\alpha,+\infty)$.
Here, we only prove our conclusion for ${\rm
dom}\phi_1=(\alpha,\beta)$. By the same method we can prove our
conclusion for other cases.  For ${\rm dom}\phi_1=(\alpha,\beta)$,
by Definition \ref{19}, it is clear that $(S\phi_1)(t)=+\infty$ for
$|t|\geq\frac{\beta-\alpha}{2}$. If $|t|<\frac{\beta-\alpha}{2}$,
let
$\lim_{x\rightarrow\alpha,\;x>\alpha}\phi_1(x)=b_1,\;\;\lim_{x\rightarrow\beta,\;x<\beta}\phi_1(x)=b_2$,
then we consider the following four cases. (a) If $b_1=b_2=+\infty$,
then by (\ref{33}), there exists some $t_1\in\mathbb{R}$ satisfying
(\ref{4}). (b) If $b_1<+\infty,\;\; b_2=+\infty$, then there exists
$\gamma\in(\alpha,\beta)$ such that $\phi_1(\gamma)=b_1$. Then by
(\ref{33}), for $|t|<\frac{\gamma-\alpha}{2}$, (\ref{4}) is
established, for $|t|\geq\frac{\gamma-\alpha}{2}$, we have
$(S\phi_1)(t)=\phi_1(\alpha+2t)\geq b_1$. (c) If
$b_1=+\infty,\;\;b_2<+\infty$, then there exists
$\gamma\in(\alpha,\beta)$ such that $\phi_1(\gamma)=b_2$. Then by
(\ref{33}), for $|t|<\frac{\beta-\gamma}{2}$, (\ref{4}) is
established, for $|t|\geq\frac{\gamma-\alpha}{2}$, we have
$(S\phi_1)(t)=\phi_1(\beta-2t)\geq b_2$. (d) If
$b_1<\infty,\;\;b_2<+\infty$, we consider three cases. If $b_1=b_2$,
then (\ref{4}) is established. If $b_1>b_2$, the proof is same as in
(c). If $b_1<b_2$, the proof is same as in (b). This completes the
proof.
\end{proof}

\section{The proofs of theorems}

In order to prove theorems stated in the introduction, we have to
establish the following six lemmas:
\begin{lem}\label{11}
If $f$ be a log-concave function such that $0<\int f<\infty$, then
the function $F$ defined by $F(z):=\int_{\mathbb{R}^n}f^{z}(x)dx$
has the following properties. {\rm(}i{\rm)} $F(z)$ is a coercive
convex function on
 $\mathbb{R}^n$ and is strictly convex on ${\rm int}\;{\rm dom}F$; {\rm(}ii{\rm)} If $f(x)$ is even about $z_0$, then $F(z)$ is also
even about $z_0$.
\end{lem}

\begin{proof}
(i) {\bf Step 1.} We shall prove $F$ is coercive. Let $f=e^{-\phi}$,
for any given $z\in\mathbb{R}^n$ and $r>0$, we have
\begin{eqnarray}\label{b23}
F(z)=\int_{\mathbb{R}^n}f^{z}(x+z)dx
\geq\int_{rB_2^n}f^{z}(x+z)dx=\int_{rB_2^n}e^{-\mathcal
{L}\phi(x)+\langle x,z\rangle}dx.
\end{eqnarray}

Since $f=e^{-\phi}$ is integrable, there is $\gamma>0$ and
$h\in\mathbb{R}$ such that
\begin{eqnarray}\label{35}
\phi(x)\geq\gamma\sum_{i=1}^{n}|x_i|+h\;\;{\rm for}\;{\rm
any}\;x\in\mathbb{R}^n.
\end{eqnarray}
Thus, for $y\in\gamma B_\infty^n$, where
$B_\infty^n=\{x\in\mathbb{R}^n:|x_i|\leq 1,i=1,\dots,n\}$, $\mathcal
{L}\phi(y)\leq\sup_{x\in\mathbb{R}^n}[\langle y,
x\rangle-\gamma\sum_{i=1}^{n}|x_i|-h]\leq -h$. Let
$rB_2^n\subset\frac{1}{2}\gamma B_\infty^n$, we have
$rB_2^n\subset{\rm int}({\rm dom}\mathcal {L}\phi)$. Since function
$g(x):\;=e^{-\mathcal {L}\phi(x)}$ is continuous on $rB_2^n$. Thus,
there exists $m>0$ such that $g(x)\geq m$ for any $x\in rB_2^n$.
Therefore,
\begin{eqnarray}\label{b26}
\int_{rB_2^n}e^{-\mathcal {L}\phi(x)+\langle x,z\rangle}dx\geq
m\int_{rB_2^n}e^{\langle x,z\rangle}dx.
\end{eqnarray}

For any $z\in\mathbb{R}^n$ and $|z|\geq 1$, let
$z^{\prime}=\frac{r}{2}\frac{z}{|z|}$, we get a closed  half-space
$H^{+}=\{x\in\mathbb{R}^n: \langle x-z^{\prime},z\rangle\geq 0\}$.
For any $x\in H^{+}$, we have $\langle x,z\rangle\geq \langle
z^{\prime}, z\rangle=\frac{r}{2}|z|$. Therefore,
\begin{eqnarray}\label{b30}\int_{rB_2^n}e^{\langle x,z\rangle}dx
&\geq&\int_{(rB_2^n)\cap H^{+}}e^{\frac{r|z|}{2}}dx=V_n((rB_2^n)\cap
H^{+})e^{\frac{r|z|}{2}}.\end{eqnarray} Since $V_n((rB_2^n)\cap
H^{+})$ is a positive constant independent of $z$, by (\ref{b23}),
(\ref{b26}) and (\ref{b30}), $F(z)$ is coercive.

 {\bf Step 2.} We shall prove that $F$ is convex and is strictly convex on ${\rm int}\;{\rm dom}F$. First, we prove $F(z)$ is proper. It is clear that $F(z)>-\infty$
for any $z\in\mathbb{R}^n$. The following claim shows that
$\{F=\infty\}\neq\mathbb{R}^n$.
\begin{cla}\label{13} For any
$z\in{\rm int}\;{\rm supp}f$, $F(z)<\infty$.
\end{cla}
\noindent{\it Proof of Claim \ref{13}.}
 For any $z\in{\rm
int}\;{\rm supp}f$, there is a closed ball $z+rB_2^n\subset{\rm
supp}f$. Since ${\rm supp}f={\rm dom}\phi$, there is
$M\in\mathbb{R}$ such that $M=\sup\{\phi(y): y\in z+rB_2^n\}$. Thus,
we have $$f^z(x)\leq \exp\{-\sup_{y\in (z+rB_2^n)}[\langle
x-z,y-z\rangle-\phi(y)]\}\leq e^M \cdot e^{-r|x-z|^2}.$$
 Therefore,
$\int_{\mathbb{R}^n} f^z(x)dx\leq e^M\int_{\mathbb{R}^n}
e^{-r|x-z|^2}dx<\infty.$ \hfill$\Box$\\

For any $z_1,z_2\in\mathbb{R}^n$ and $\alpha\in (0,1)$. Let
$f=e^{-\phi}$, we have $F(z)=\int_{\mathbb{R}^n}e^{-\mathcal
{L}\phi(x)+\langle x,z\rangle}dx$. Since $g_x(z):=e^{-\mathcal
{L}\phi(x)+\langle x,z\rangle}$ is a convex function about $z$, we
have
\begin{eqnarray}\label{b37}
F(\alpha z_1+(1-\alpha)z_2)\leq\alpha F(z_1)+(1-\alpha)F(z_2).
\end{eqnarray}

If $z_1,z_2\in{\rm int}\;{\rm dom}F$ and $z_1\neq z_2$, then
inequality (\ref{b37}) is a strict inequality. Thus $F(z)$ is
strictly convex on ${\rm int}\;{\rm dom}F$.

(ii) Since $f(x)$ is even about $z_0$, $f(z_0+x)=f(z_0-x)$ for any
$x\in\mathbb{R}^n$. For any $z\in \mathbb{R}^n$, we have $$F(z_0+z)
=\int_{\mathbb{R}^n}f^{z_0+z}(x)dx=\int_{\mathbb{R}^n}f^{z_0-z}(-x+2z_0)dx
=F(z_0-z).$$ This completes the proof.
\end{proof}
\begin{rem} By Lemma \ref{11}, if $f$ is even about $z_0$, then $s(f)=z_0$.
\end{rem}

\begin{lem} \label{68}Let $f$ be
a log-concave function such that $0<\int f<\infty$, and let
$G\subset\mathbb{R}^n$ be an affine subspace satisfying $G\cap{\rm
int}\;{\rm supp}f\neq\emptyset$. Then there exists a unique point
$z_0\in G$ satisfying the following two equivalent claims. (i)
$F(z_0)=\min\{F(z); z\in G\}$, where
$F(z):=\int_{\mathbb{R}^n}f^{z}(x)dx$. (ii) ${\rm grad}
F(z_0)=\int_{\mathbb{R}^n} x f^{z_0}(x+z_0)dx\in G^{\bot}$.
\end{lem}
\begin{proof} By Lemma \ref{11}, $F$ is coercive and strictly convex on ${\rm int}\;{\rm dom}F$, thus there is a
unique minimal point $z_0=s_G(f)$ on $G$. Let $f=e^{-\phi}$, then
$F(z)=\int_{\mathbb{R}^n}e^{-\mathcal {L}\phi(x)+\langle
x,z\rangle}dx$. By the dominated convergence theorem, we have ${\rm
grad} F(z)=\int_{\mathbb{R}^n}xe^{-\mathcal {L}\phi(x)+\langle
x,z\rangle}dx=\int_{\mathbb{R}^n}xf^z(x+z)dx$.

Next, we prove the equivalence of  (i) and (ii). Let
$\eta_1,\dots,\eta_m\;(m<n)$ be an orthonormal basis of $G$ and let
$\eta_{m+1},\dots,\eta_n$ be an orthonormal basis of $G^{\perp}$.
Let $z=\sum_{i=1}^{n}z_i\eta_i$, since $z_0=s_G(f)\in G$, we have
$\left.\frac{\partial F(z)}{\partial z_i}\right|_{z=z_0}=
\lim_{t\rightarrow
0}\frac{F(z_0+t\eta_i)-F(z_0)}{t}=0,\;\;i=1,\dots,m$. Hence, ${\rm
grad} F(z_0)\in G^{\bot}$. On the other hand, if ${\rm grad}
F(z_0)\in G^{\bot}$, then $\left.\frac{\partial F(z)}{\partial
z_i}\right|_{z=z_0}=0,\;i=1,\dots,m$. Since $F(z)$ is strictly
convex on $G\cap{\rm int}\;{\rm dom}F$, $z_0$ is the unique minimal
point on $G$.
\end{proof}
\begin{rem}
In Lemma \ref{68}, if $G=\mathbb{R}^n$, then the lemma shows that
the Santal\'{o} point $s(f)$ of $f$ is the barycenter of the
function $f^{s(f)}$.
\end{rem}

\begin{lem}\label{73}
Let $f$ be a log-concave function such $0<\int f<\infty$. Let
$G\subset\mathbb{R}^n$ be an affine subspace  satisfying $G\cap{\rm
int}\;{\rm supp}f\neq\emptyset$ and $z=s_G(f)$. Let $H$ be an affine
hyperplane such that $G\subset H$ and let $g$ be the function
defined by $g^{z}=S_H (f^{z})$. Then we have $s_G(g)=z=s_G(f)$.
\end{lem}
\begin{proof} It may be supposed that $z=s_G(f)=0$,
$H=\{(x_1,\cdots,x_n)\in\mathbb{R}^n:x_n=0\}$ and
$G=\{(x_1,\cdots,x_n)\in\mathbb{R}^n:x_{m+1}=\cdots=x_{n}=0\}$ for
some $m$, $1\leq m\leq n-1$. By Lemma \ref{68}, we have
$\int_{\mathbb{R}^n}xf^{0}(x)dx\in G^{\bot}$. Let
$f^{0}_{x^{\prime}}(t):=f^{0}(x^{\prime}+tu)$ for any $x^{\prime}\in
H$, where $u$ is the unit normal vector of $H$. Thus,
$\int_{H}x_i\left(\int_{\mathbb{R}}f^{0}_{x^{\prime}}(t)dt\right)
dx^{\prime}=0\;\;\textrm{for} \;\;1\leq i\leq m$. By $g^{0}=S_H
(f^{0})$ and (\ref{21}), for every $x^{\prime}\in H$,
$\int_{\mathbb{R}}f^{0}_{x^{\prime}}(t)=\int_{\mathbb{R}}g^0_{x^{\prime}}(t)$.
Thus,
$\int_{H}x_i\left(\int_{\mathbb{R}}g^{0}_{x^{\prime}}(t)dt\right)
dx^{\prime}=0\;\;\textrm{for} \;\;1\leq i\leq m$, which conversely
gives $\int_{\mathbb{R}^n}x g^{0}(x)dx\in G^{\bot}$. Thus, by Lemma
\ref{68} again, we obtain $s_G(g)=0=s_G(f)$.
\end{proof}

\begin{lem}\label{71}
For a log-concave function $f$ such that $0<\int f<\infty$, if $f$
is symmetric about some affine hyperplane $H$, then, for any $z\in
H$, $f^{z}$ is also symmetric about $H$.
\end{lem}
\begin{proof} Let $u$ be the unit normal vector of $H$. For any
$x^{\prime},y^{\prime}\in H$ and $s,t\in \mathbb{R}$, since
$f(x^{\prime}+su)=f(x^{\prime}-su)$, we have
\begin{eqnarray*}\label{72}
f^{z}(y^{\prime}+tu)&=& \inf_{x^{\prime}+su\in
\mathbb{R}^n}\frac{\exp\{-\langle y^{\prime}+tu-z,
x^{\prime}+su-z\rangle\}}{f(x^{\prime}+su)}
\nonumber\\
&=&\inf_{x^{\prime}+su\in \mathbb{R}^n}\frac{\exp\{-\langle
y^{\prime}-z-tu, x^{\prime}-z-su\rangle\}}{f(x^{\prime}-su)}=
f^{z}(y^{\prime}-tu).
\end{eqnarray*}
This completes the proof.
\end{proof}

\begin{lem}\label{79} Let $f$ be a
log-concave function such that $0<\int f<\infty$ and let $H$ be an
affine hyperplane satisfying $H\cap{\rm int}\;{\rm
supp}f\neq\emptyset$ and $z\in H\cap{\rm int}\;{\rm supp}f$; let
$\lambda$, $0<\lambda<1$ such that $H$ is $\lambda$-separating for
$f^{z}$. Then
$$\int_{\mathbb{R}^n}(S_H f)^{z}\geq 4\lambda(1-\lambda)\int_{\mathbb{R}^n}f^{z}.$$
\end{lem}
\begin{proof} It may be supposed that $z=0$ and $H=\{(x_1,\dots,x_n):
x_n=0\}$. For $y^{\prime}\in H$ and $s\in\mathbb{R}$, let
$(y^{\prime},s)$ denote $y^{\prime}+su$, where $u$ is a unit normal
vector of $H$. For $f^0$ and $s\in\mathbb{R}$, we define a new
function
$$f^0_{(s)}(y^{\prime}):=f^0(y^{\prime},s),\;{\rm for\; any}\;y^{\prime}\in H.$$

 Next we shall prove that for any $y^{\prime}\in H$ and $s,t>0$
\begin{eqnarray}\label{80}
\left(\frac{t}{s+t}\cdot
f_{(s)}^{0}\right)\star\left(\frac{s}{s+t}\cdot
f_{(-t)}^{0}\right)(y^{\prime})\leq (S_H f)^{0}_{
(\frac{2st}{s+t})}(y^{\prime}).
\end{eqnarray}
\begin{cla}\label{15}
For any $x^{\prime}\in H$ and $w\in\mathbb{R}$, if $(S_H
f)(x^{\prime}+wu)>0$, then there is some $w_1\in\mathbb{R}$ such
that $(S_H f)(x^{\prime}+wu)\leq f(x^{\prime}+w_1u)$ and $(S_H
f)(x^{\prime}+wu)\leq f(x^{\prime}+(w_1-2w)u)$.
\end{cla}
\noindent{\it Proof of Claim \ref{15}.} Let $f=e^{-\phi}$, since
$(S_H f)(x^{\prime}+wu)>0$, then $(S_H\phi)(x^{\prime}+wu)<+\infty$.
By Proposition \ref{7}(iii), there is $w_1\in\mathbb{R}$ such that
$(S_H\phi)(x^{\prime}+wu)\geq \phi(x^{\prime}+w_1u)$ and
$(S_H\phi)(x^{\prime}+wu)\geq \phi(x^{\prime}+(w_1-2w)u)$, here we
assume $\phi(x^{\prime}+w_1u)$ or $\phi(x^{\prime}+(w_1-2w)u)$
equals the limit in Proposition \ref{7}(iii), which doesn't affect
our proof. Hence the claim follows.\hfill$\Box$\\
\indent For any $y_1^{\prime}$, $y_2^{\prime}\in H$ such that
$y^{\prime}=y_1^{\prime}+y_2^{\prime}$, we have
\begin{eqnarray*}
(S_H f)_{ (\frac{2st}{s+t})}^{0}(y^{\prime})
&=&\inf_{(x^{\prime},w)\in
H\times\mathbb{R}}\frac{\exp\{-\langle(y^{\prime},\frac{2st}{s+t}),(x^{\prime},w)\rangle\}}{(S_H
f)(x^{\prime},w)}\nonumber\\
&\geq&\inf_{(x^{\prime},w)\in
H\times\mathbb{R}}\frac{\exp\{-\langle(y^{\prime},\frac{2st}{s+t}),(x^{\prime},w)\rangle\}}{f(x^{\prime},w_1)^{\frac{t}{s+t}}f(x^{\prime},w_1-2w)^{\frac{s}{s+t}}}\nonumber\\
&\geq&\inf_{(x^{\prime},w)\in
H\times\mathbb{R}}\frac{\exp\{-\frac{t}{s+t}\langle(\frac{s+t}{t}y_1^{\prime},s),(x^{\prime},w_1)\rangle\}}{f(x^{\prime},w_1)^{\frac{t}{s+t}}}\nonumber\\
&&\times\inf_{(x^{\prime},w)\in
H\times\mathbb{R}}\frac{\exp\{-\frac{s}{s+t}\langle(\frac{s+t}{s}y_2^{\prime},-t),(x^{\prime},w_1-2w)\rangle\}}{f(x^{\prime},w_1-2w)^{\frac{s}{s+t}}}\nonumber\\
&\geq&f^0\left(\frac{s+t}{t}y_1^{\prime},s\right)^{\frac{t}{s+t}}f^0\left(\frac{s+t}{s}y_2^{\prime},-t\right)^{\frac{s}{s+t}},
\end{eqnarray*}
where the first inequality is by Claim \ref{15}, and the second
inequality is by $\inf(AB)\geq(\inf A)(\inf B)$, and last inequality
is by the definition of the polar of functions. Since $y^{\prime}_1$
and $y^{\prime}_2$ are arbitrary, we get (\ref{80}).

Let $F_0(w)=\int_{H}(S_H f)^{0}_{(w)}$, $F_1(s)=\int_{H}f^{0}_{(s)}$
and $F_2(t)=\int_{H}f^{0}_{(-t)}$. By the Pr\'{e}kopa
inequality and (\ref{80}), we have $$F_0(\frac{2st}{s+t})\geq
F_1(s)^{\frac{t}{s+t}}F_2(t)^{\frac{s}{s+t}}\;{\rm for}\;{\rm
every}\;s,t>0.$$ Now, by Proposition \ref{7}(i) and Lemma \ref{71},
$(S_H f)^{0}$ is symmetric about $H$, we have
$\int_{0}^{+\infty}F_0=\frac{1}{2}\int_{\mathbb{R}^n}(S_H f)^0$ and
since $H$ is $\lambda$-separating for $f^{0}$, we have
$\left(\int_{0}^{+\infty}F_1\right)\left(\int_{0}^{+\infty}F_2\right)=\lambda(1-\lambda)\left(\int_{\mathbb{R}^n}f^0\right)^2$.
Since $F_0$, $F_1$, $F_2:[0,+\infty)\rightarrow \mathbb{R}^{+}$
satisfy the hypothesis of Lemma \ref{74}, and by definitions of
$F_1$ and $F_2$, one has
$\int_{0}^{+\infty}F_1+\int_{0}^{+\infty}F_2=\int_{\mathbb{R}^n}f^0$,
thus, by Lemma \ref{74}
\begin{eqnarray*}\label{86}
\frac{2}{\int_{\mathbb{R}^n}(S_H
f)^0}\leq\frac{1}{2}\left(\frac{1}{\int_{0}^{+\infty}F_1}+\frac{1}{\int_{0}^{+\infty}F_2}\right)=\frac{1}{2\lambda(1-\lambda)\int_{\mathbb{R}^n}f^0}.
\end{eqnarray*}
This gives the desired inequality.
\end{proof}

\begin{lem} \label{77} If $f$ is an integrable, unconditional, log-concave function, then
$\int_{\mathbb{R}^n}f\int_{\mathbb{R}^n}f^{0}\leq (2\pi)^n$.
\end{lem}
\begin{proof} Let $f_1=f$, $f_2=f^{0}$ and $f_3=e^{-\frac{|x|^2}{2}}$, then $f_1$, $f_2$ and
$f_3$ are unconditional. Thus we have
$\int_{\mathbb{R}^n}f_j=2^n\int_{\mathbb{R}_+^n}f_j,\;\;j=1,2,3$.
For $(y_1,\dots,y_n)\in\mathbb{R}^n$, we define
$g_i(y_1,\dots,y_n)=f_i(e^{y_1},\dots,e^{y_n})e^{\sum_{i=1}^{n}y_i}$.
We get $\int_{\mathbb{R}_+^n}f_j=\int_{\mathbb{R}^n}g_j$, and for
every $s,t\in\mathbb{R}^n$, $g_1(s)g_2(t)\leq
g_3\left(\frac{s+t}{2}\right)^2$. Hence
$\int_{\mathbb{R}^n}f\int_{\mathbb{R}^n}f^{0}\leq (2\pi)^n$ follows
from Pr\'{e}kopa inequality.
\end{proof}

\noindent{\it Proof of Theorem \ref{2}.} We proceed by $n$
successive Steiner symmetrizations until we get an unconditional
log-concave function.

Let $u_1\in S^{n-1}$, $u_1$ orthogonal to $H=H_1$ and let
$(u_i)_{i=2}^{n}\subset S^{n-1}$ such that $(u_1,\dots,u_n)$ form an
orthonormal basis for $\mathbb{R}^n$. Let $z_1=s_{H_1}(f)$ and
define a log-concave function $f_1$ by the identity
$f_1^{z_1}=S_{H_1}(f^{z_1})$. Then $\int f_1^{z_1}=\int f^{z_1}$. By
Proposition \ref{7}(i) and Lemma \ref{71}, $f_1$ is symmetric about
$H_1$ and by Lemma \ref{79}, applied to $f^{z_1}$, $z=z_1$ and
$H=H_1$, $\lambda$-separating for $f=(f^{z_1})^{z_1}$, we get
$\int_{\mathbb{R}^n}f_1\geq 4\lambda(1-\lambda)\int_{\mathbb{R}^n}f$
and thus $\int f_1\int f_1^{z_1}\geq 4\lambda(1-\lambda)\int f\int
f^{z_1}$. Choose now the hyperplane $H_2$, orthogonal to $u_2$, and
medial for $f_1$ and define $z_2=s_{(H_1\cap H_2)}(f_1)$. By Lemma
\ref{73} we have $z_1=s_{H_1}(f)=s_{H_1}(f_1)$, we get $\int
f_1^{z_2}=\min_{z\in H_1\cap H_2}\int f_1^{z}\geq\min_{z\in H_1}\int
f_1^{z}=\int f_1^{z_1}$. We define now a new log-concave function
$f_2$ by the identity $f_2^{z_2}=S_{H_2}(f_1^{z_2})$. By Proposition
\ref{7}(ii) and Lemma \ref{71}, $f_2$ is symmetric about both $H_1$
and $H_2$. Since $H_2$ is medial for $f_1$, we get by Lemma \ref{79}
applied to $f_1^{z_2}$, $z=z_2$ and $H=H_2$ that $\int f_2\geq\int
f_1$. Moreover, we have $\int f_2^{z_2}=\int S_{H_2}(f_1^{z_2})=\int
f_1^{z_2}\geq\int f_1^{z_1}$. It follows that $\int f_2\int
f_2^{z_2}\geq\int f_1\int f_1^{z_1}$.

We continue this procedure by choosing hyperplanes $H_2,\dots,H_n$,
points $z_2,\dots,z_n$, and defining log-concave functions
$f_2,\dots,f_n$ such that for $2\leq i\leq n$, we have (i) $H_i$ is
medial for $f_{i-1}$ and orthogonal to $u_i$; (ii) $z_i=s_{(H_1\cap
H_2\cap\dots\cap H_i)}(f_{i-1})$; (iii)
$f_i^{z_i}=S_{H_i}(f_{i-1}^{z_i})$.  From (ii) (iii) and Lemma
\ref{73}, we have $z_i=s_{(H_1\cap\dots\cap
H_i)}(f_{i-1})=s_{(H_1\cap\dots\cap H_i)}(f_i)$. Choosing $H_{i+1}$,
$z_{i+1}$, $f_{i+1}$ according to (i) (ii) (iii), we get thus $\int
f_{i+1}^{z_{i+1}}=\int S_{H_{i+1}}(f_{i}^{z_{i+1}})=\int
f_{i}^{z_{i+1}} \geq \int f_i^{s_{(H_1\cap\dots\cap H_i)}(f_i)}=\int
f_i^{z_i}$. Now, Lemma \ref{79} applied to $f_i^{z_{i+1}}$,
$z=z_{i+1}$ and $H_{i+1}$, medial for
$f_i=(f_i^{z_{i+1}})^{z_{i+1}}$, gives $\int f_{i+1}\geq\int f_i$.
Thus, $\int f_i\int f_i^{z_i}$ is an increasing sequence, for $2\leq
i\leq n$.
 Therefore, we have
$4\lambda(1-\lambda)\int f\int f^{z_1} \leq\int f_1\int
f_1^{z_1}\leq\dots \leq\int f_n\int f_n^{z_n}$.
 From
Proposition \ref{7}(ii), $f_n$ is an unconditional function about
$z_n$ and $z_n\in H_1\cap H_2\cap\dots\cap H_n$ is a center of
symmetry for $f_n$. By Lemma \ref{77}, we have $\int f\int
f^{z_1}\leq \frac{(2\pi)^n}{4\lambda(1-\lambda)}$,
this concludes the proof.\hfill$\Box$\\

\bibliographystyle{amsalpha}

\end{document}